\documentclass{amsart}
\usepackage{amsmath}%
\usepackage{amsfonts}%
\usepackage{amssymb}%
\usepackage{graphicx}
\usepackage{mathcomp}
\usepackage{mathrsfs}
\usepackage{euscript}

\usepackage[utf8]{inputenc}
\usepackage{subfigure}
\usepackage{amssymb}
\usepackage{amsfonts}
\usepackage{latexsym}
\usepackage{amsmath}

\usepackage{mathcomp}
\usepackage{mathrsfs}
\usepackage{euscript}

\usepackage{graphics}
\usepackage{graphicx}
\usepackage{caption}
\usepackage{booktabs}
\usepackage{floatrow} 
\usepackage{subfig}
\usepackage{psfrag}
\usepackage{epic,eepic}
\usepackage{color}
\usepackage[table]{xcolor}
\usepackage{tikz}

\usepackage{cite}
\usepackage{enumerate}
\usepackage{paralist}
\usepackage{amsrefs}   
\usepackage{fullpage}

\usepackage[left, mathlines]{lineno}
\usepackage{float}
\usepackage{mathtools}

\usepackage{array}

\def\qq{{\mathcal Q}}
\def\pp{{\mathcal P}}

\def\tt{{\mathcal T}}

\def\RR{{\mathbb R}}

\theoremstyle{plain}
\newtheorem{thm}{Theorem}
\newtheorem{lem}[thm]{Lemma}
\newtheorem{cor}[thm]{Corollary}
\newtheorem{prop}[thm]{Proposition}

\newtheorem{obs}[thm]{Observation}
\newtheorem{stat}[thm]{Statement}

\theoremstyle{definition}

\theoremstyle{plain}

\numberwithin{equation}{section}

\begin{document}


\title{Disjointness graphs of segments in ${\mathbb R}^2$ are almost all hamiltonian}

\author{J. Lea\~nos$^1$ \and Christophe Ndjatchi$^{2}$ \and L. M. R\'ios-Castro$^{3}$ }
\maketitle{}

\noindent\footnote{
             Unidad Acad\'emica de Matem\'aticas, Universidad Aut\'onoma de Zacatecas, M\'exico.
             Email: jleanos@uaz.edu.mx       \\
             $^{2}$Academia de F\'isico-Matem\'aticas, Instituto Polit\'ecnico Nacional, UPIIZ,
  P.C. 098160, Zacatecas, M\'exico. \\Email: mndjatchi@ipn.mx \\
            $^{3}$Academia de F\'isico-Matem\'aticas, Instituto Polit\'ecnico Nacional, CECYT18, Zacatecas,
  P.C. 098160, Zacatecas, M\'exico. \\Email: lriosc@ipn.mx\\}

\begin{abstract}
Let $P$ be a set of  $n\geq 2$ points in general position in ${\mathbb R}^2$. The {\em edge disjointness graph} $D(P)$ of $P$ is the graph whose
vertices are all the closed straight  line segments with endpoints in $P$, two of which are adjacent in $D(P)$ if and only if they are disjoint.
In this note, we give a full characterization of all those edge disjointness graphs that are hamiltonian. More precisely, we shall show that (up to order type isomorphism)
there are exactly 8 instances of $P$ for which $D(P)$ is not hamiltonian. Additionally, from one of these 8 instances, we derive a counterexample to a
 criterion for the existence of hamiltonian cycles due to A. D. Plotnikov in 1998.
\end{abstract}

{\it Keybwords:} Disjointness graph of segments; Hamiltonian cycles; Rectilinear drawings of complete graphs.

\section{Introduction}

Let $P$ be a set of $n\geq 2$ points in general position in the plane, i.e., no three points in $P$ are collinear.
A {\em segment} of $P$ is a closed straight line segment with its two endpoints being elements of $P$. In this note, we shall use $\pp$
 to denote the set of all $\binom{n}{2}$ segments of $P$. The  {\em edge disjointness graph}
 $D(P)$ of $P$ is the graph whose vertex set is $\pp$, and two segments of $\pp$ are adjacent in $D(P)$ if and only if
 they are disjoint.

 If $x$ and $y$ are distinct points of $P$, we shall use $xy$ to denote the closed straight line segment whose endpoints are $x$ and $y$.
 We often make no distinction between an element of $\pp$ and its corresponding vertex in $D(P)$. As usual, we will denote by $CH(P)$ the boundary of the
 convex hull of $P$,  and by $\overline{P}$ to $P\cap CH(P)$. Then, if $P$ is in convex position, we have $P=\overline{P}$. See Figure~\ref{fig:Ejem}.


The edge disjointness graphs were introduced in 2005 by Araujo, Dumitrescu, Hurtado, Noy, and Urrutia~\cite{gaby}, as geometric versions of the Kneser graphs. We recall that for $k,m\in {\mathbb Z}^+$ with $k\leq  m/2$, the \emph{Kneser graph} $KG(m; k)$ is defined as the graph whose vertices are all the $k$--subsets of $\{1,2,\ldots ,m\}$ and in which two $k$-subsets form an edge if and only if they are disjoint. In~\cite{kneser} Kneser conjectured that the chromatic number $\chi(KG(m; k))$ of $KG(m; k)$ is equal to $m-2k+2$. This conjecture was proved by Lov\'asz~\cite{lovasz} in 1978 and, as the reader can check in~\cite{chen,baptist,matousek} and the references therein,
the research on Kneser graphs is still of interest.

The chromatic number $\chi(D(P))$ of $D(P)$ has been studied in~\cite{gaby,lomeli,jonsson,ruy}. As far as we know, the determination of the exact value of $\chi(D(P))$ remains open in general. In~\cite{gaby} a general lower bound for $\chi(D(P))$ was established.  On the other hand,  the exact value of
$\chi(D(P))$ is known only for two families of point sets: when $P$ is in convex position~\cite{ruy,jonsson}, and when $P$ is a double chain~\cite{lomeli}.
In ~\cite{pach-tardos-toth} Pach, Tardos, and T\'oth studied the chromatic number and the clique number of $D(P)$ in the more general setting of $\RR^d$ for $d\geq 2$, i.e., when $P\subset \RR^d$. More precisely, in~\cite{pach-tardos-toth} it was shown that the chromatic number of $D(P)$ is bounded by above by a polynomial function that depends on its clique number $\omega(D(P))$,  and that the problem of determining any of $\chi(D(P))$ or $\omega(D(P))$ is NP-hard. More recently, Pach and Tomon~\cite{pach-tomon} have shown that if $G$ is the disjointness graph of certain types of curves (namely, grounded $x$-monotone curves) in $\RR^2$ and $\omega(G)=k$, then $\chi(G)\leq k+1$.
 We remark that the grounded $x$-monotone curves in~\cite{pach-tomon}  play the role of the straight line segments in $\pp$, i.e., two grounded $x$-monotone curves are adjacent in $G$ iff they are disjoint.

The exact determination of the connectivity number $\kappa(D(P))$ of $D(P)$ is another open problem. In~\cite{us2} and~\cite{us1} the following upper and lower bounds have been reported, respectively. $$\binom{\lfloor\frac{n-2}{2}\rfloor}{2}+\binom{\lceil\frac{n-2}{2}\rceil}{2} \leq \kappa(D(P))\leq \frac{7n^2}{18}+\Theta(n).$$

We recall that a subset $\qq\subseteq \pp$ is {\em independent} in $D(P)$, if no two elements of $\qq$ are adjacent in $D(P)$. In~\cite{birgit}, Aichholzer, Kyn\v{c}l, Scheucher, and Vogtenhuber have established an asymptotic upper bound for the maximum size of a certain class of independent sets of $D(P)$.
 \begin{figure}[H]
	\centering
	\includegraphics[width=.65\textwidth]{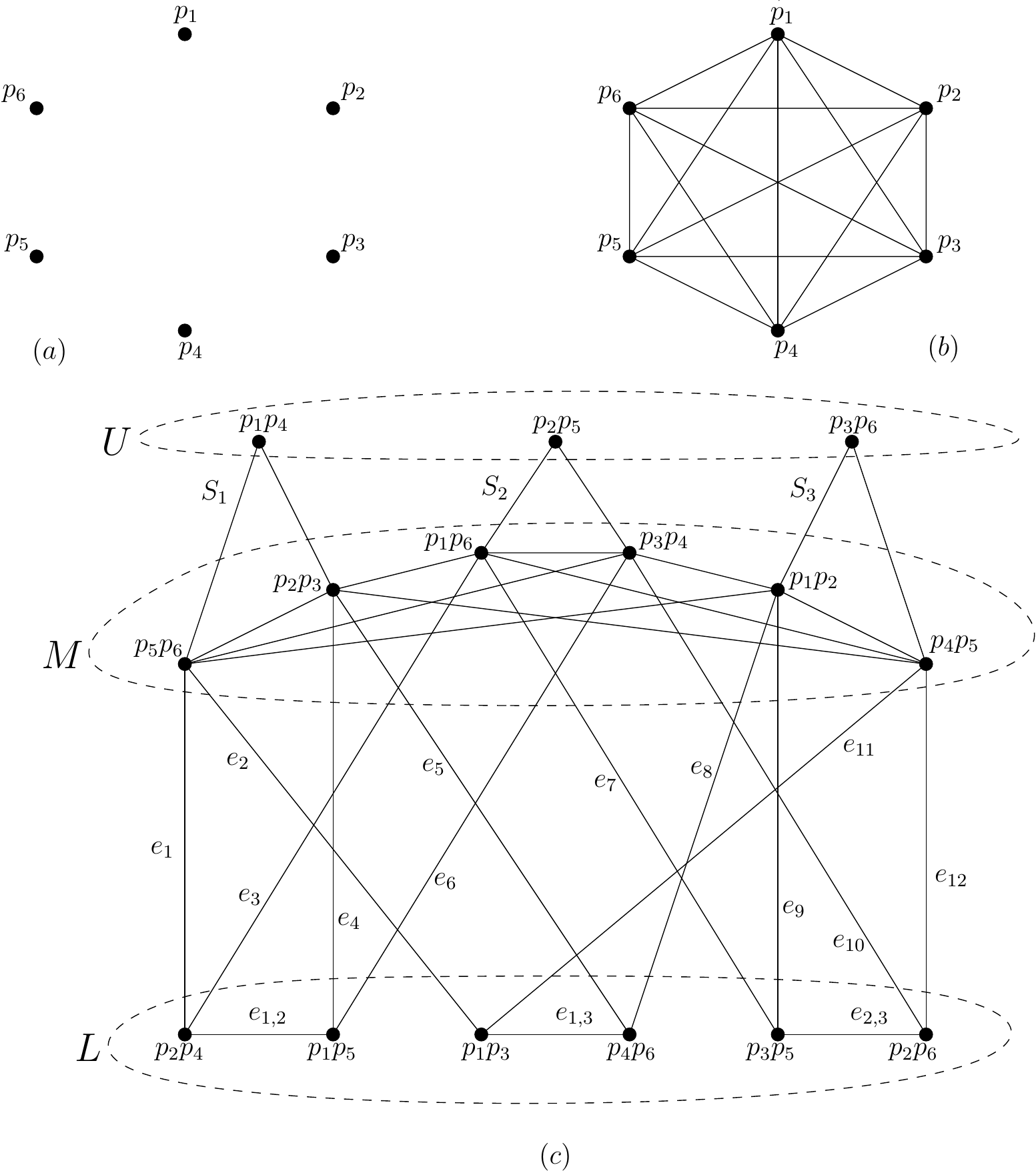}
	\caption{\small ($a$) The set $\{p_1,p_2,\ldots ,p_6\}$ of points in general (and convex) position is $P$. We note that $\overline{P}=P$. In (b) we have the corresponding
	$\pp$, which can be seen as a rectilinear drawing of the complete graph $K_6$. Note that $CH(P)$ is the convex polygon formed by the segments $p_1p_2, p_2p_3, p_3p_4, p_4p_5, p_5p_6,$ and $p_6p_1$. Finally, the graph in (c) is the edge disjointness graph $D(P)$ defined by $P$.}
\label{fig:Ejem}
\end{figure}
Let $H=(V(H),E(H))$ be a simple connected graph.  We recall that a {\em cycle} $C$ of $H$ is a sequence
$C:=v_0, v_1, \ldots , v_{k-1}, v_k, v_0$ satisfying the following:
\begin{itemize}
\item $k$ is an integer such that $2\leq k< |V(H)|$,
\item $\{v_0, v_1,\ldots ,v_k\}\subseteq V(H)$,
\item $v_iv_{i+1}\in E(H)$ for each $i\in \{0,1,\ldots ,k\}$, where $i+1$ is taken$\mod k$.
\item $v_i\neq v_j$ for distinct $i,j\in \{0,1,\ldots ,k\}$,
\end{itemize}

The graph $H$ is called {\em hamiltonian} if it contains a {\em spanning} cycle, i.e. a cycle passing through each vertex of $H$.
To determine whether or not a given connected graph $H$ is hamiltonian is a classical problem in graph theory that remains tantalizingly open. See~\cite{gould}
for a good recent survey on this problem.
In particular, it is well-known that  such a problem is NP-complete~\cite{garey-johnson}. In view of this fact, many standard approaches have been developed for the study of
hamiltonicity of graphs. For instance, the determination of necessary and/or sufficient conditions under which a graph contains a spanning cycle is one of them.
In this note, we follow another common approach, which can be written as follows:  {\em given a nontrivial family of graphs ${\mathcal G}$, provide a characterization of
all those elements of ${\mathcal G}$ that are hamiltonian.}  In this work, our interest lies in the following family.
$${\mathcal G}=\{D(P)~:~P \mbox{ is a finite point set in general position of } {\mathbb R}^2  \mbox{ with } |P|\geq 2\}.$$

Let $a= (x_1, y_1), b= (x_2, y_2)$ and $c= (x_3, y_3)$ be points in ${\mathbb R}^2$. The {\em orientation} of the ordered triple $abc$ is the sign of $O(abc)$, where
\[O(abc):=\left|
\begin{array}{ccc}

1 & x_1 & y_1 \\

1 & x_2 & y_2 \\

1 & x_3 & y_3
\end{array}
\right|.
\]
Let $P$ and $Q$ be two finite point sets in general position in ${\mathbb R}^2$. These point sets have the same {\em order type} if and only if there is a bijection
$\gamma:P\to Q$ such that each ordered triple $abc$ in $P$ has the same orientation as its image $\gamma(a)\gamma(b)\gamma(c)$.
We write $P\sim Q$ if $P$ and $Q$ have the same order type, and $P\not\sim Q$ otherwise. In particular, it is well known that
$\sim$ defines an equivalence relation on the finite point sets in general position in ${\mathbb R}^2$.

Note that if $\gamma:P\to Q$ is a bijection, then $\gamma$ can be naturally extended to a bijection $\Gamma: \pp\to \qq$. Moreover, if $\gamma$ verifies that $P\sim Q$,
then (see for instance~\cite{3-symmetric}) two segments $e$ and $e'$ of $\pp$ are disjoint if and only if their corresponding segments $\Gamma(e)$ and
$\Gamma(e')$ are disjoint, and so $P\sim Q$ implies that $D(P)$ and $D(Q)$ are isomorphic. In view of this, we need only consider a representative
of each equivalence class of $\sim$.

Our main result is the following.

\begin{thm}\label{thm:main}
Let $P$ be a set of $n\geq 2$ points in general position in the plane, and let $C_6$ denote the set of 6 points in convex position in ${\mathbb R}^2$. Then,
$D(P)$ is hamiltonian if and only if $n\geq 6$ and $P\not\sim C_6$.
\end{thm}

For $n\in {\mathbb Z}^+$, let $\#ot(n)$ denote the number of order types on $n$ points in general position in ${\mathbb R}^2$.
As can be checked in~\cite{oswin1},  the exact values of $\#ot(n)$ for $n=2,3,\ldots ,8$ are the given in the following table.

\begin{table}[h!]
\centering
\begin{tabular}{|c c c c c c c c|}
 \hline
  $n$ & 2 & 3 & 4 & 5 & 6 & 7 & 8 \\
  \hline
 $\#ot(n)$  & 1 &1 & 2 & 3 & 16 & 135 & 3315 \\  [1ex]
 \hline
\end{tabular}
\caption{The exact value of $\#ot(n)$ for $n\in \{2,3,\ldots ,8\}$.}
\label{table1}
\end{table}

The next corollary follows directly from Theorem~\ref{thm:main} and Table~\ref{table1}.

\begin{cor}\label{cor:main}
Let $P$ be a set of $n\geq 2$ points in general position in the plane. There are exactly 8 (order types) instances of $P$ for which
$D(P)$ is not hamiltonian.
\end{cor}


\section{The proof of Theorem~\ref{thm:main}}\label{sec:main}
Throughout this section $P$ and $\pp$ denote a set of $n\geq 2$ points in general position in ${\mathbb R}^2$ and the set of all $\binom{n}{2}$ segments of $P$, respectively.

The proof of the next proposition is an easy exercise.
 \begin{prop}\label{p:ge5}
$D(P)$ is connected if and only if $n\geq 5$.
\end{prop}

In view of Proposition~\ref{p:ge5}, for the rest of the note we can assume that $n\geq 5$.
We analyze three cases separately, depending on the size of $n$ and the structure of $P$.

\subsection{$D(P)$ is hamiltonian if $n\geq 9$}\label{subsec:nine}
A {\em thrackle} of $\pp$ is a set $\tt\subseteq \pp$ such that each pair of segments of $\tt$
either has a common endpoint or crosses properly. The following proposition was shown in~\cite{kupitz}.

\begin{prop}\label{p:max-track}
The maximum number of segments in a thrackle of $\pp$ is either $n-1$ or $n$.
\end{prop}

The next assertion follows immediately from the definitions of $D(P)$ and thrackle of $\pp$.

\begin{obs}\label{o:thrack}
 A set of vertices in $D(P)$ is independent if and only if its corresponding set of segments in $\pp$ forms a thrackle.
\end{obs}

In view of Proposition~\ref{p:max-track} and Observation~\ref{o:thrack} we have the following corollary.

\begin{cor}\label{p:independence}
The independence number $\alpha(D(P))$ of $D(P)$ is either $n-1$ or $n$.
\end{cor}

Recently Lea\~nos, Ndjatchi, and R\'ios-Castro~\cite{us1} have established the following lower bound for the
connec\-tivity number $\kappa(D(P))$ of $D(P)$.

\begin{thm}[Theorem~1~\cite{us1}]\label{t:upconec}
If $P$ is any set of $n\geq 3$ points in general position in the plane, then
$$\kappa(D(P))\geq \binom{\lfloor\frac{n-2}{2}\rfloor}{2}+\binom{\lceil\frac{n-2}{2}\rceil}{2}.$$
\end{thm}

A classical result of Chv\'atal and Erd\H{o}s in the field of Hamilton cycles is the following.

\begin{thm}[Theorem~1~\cite{erdos}]\label{t:upbounindep}
If $\alpha(G)\leq \kappa(G)$ for a graph $G$, then $G$ is hamiltonian.
\end{thm}

We note that if $n\geq 9$, then Theorems~\ref{t:upconec},~\ref{t:upbounindep} and Corollary~\ref{p:independence} imply that $D(P)$ is hamiltonian.
So we have proved the following.

\begin{lem}\label{l:nine}
If $n\geq 9$, then $D(P)$ is hamiltonian.
\end{lem}


\subsection{The instances of $P$ for which $D(P)$ is not hamiltonian}\label{subsec:negative}
\vskip 0.2cm

Our aim in this section is to show the following lemma.

\begin{lem}\label{l:negative}
$D(P)$ is not hamiltonian if at least one of the following two conditions holds:
\begin{itemize}
\item[(C1)] $n\in \{2,3,4,5\}$;
\item[(C2)] $n=6$ and $P$ is in convex position.
\end{itemize}
\end{lem}

\begin{proof} For brevity, let $G:=D(P)$. Suppose that $(C1)$ holds. From Proposition~\ref{p:ge5} we know that $G$ is not connected (and hence not hamiltonian) for $n\in \{2,3,4\}$.
Then we may assume that $n=5$. It is folklore that $P$ contains a subset $Q$ with $4$ points in convex position. Let $p_1, p_2, p_3, p_4$ be the points of $Q$, and assume that they appear in $CH(Q)$ in this cyclic order. For $i\in \{1,2\}$, let $d_i:=p_{i}p_{i+2}$. Then, $d_i$ intersects all segments of $\qq$, and so each neighbor of $d_i$ in $G$ is incident with the only point $p_5$ belonging to $P\setminus Q$. We note that regardless of whether $p_5$ lies inside of $CH(Q)$ or not, we have that
$N_G(d_1)\subseteq \{p_2p_5, p_4p_5\}$ and $N_G(d_2)\subseteq \{p_1p_5, p_3p_5\}$, where $N_G(d_i)$ denotes the set of all vertices of $G$ that are adjacent to $d_i$.

Seeking a contradiction, suppose that $G$ contains a spanning cycle $C$. From the existence of $C$ it follows that $N_G(d_1)=\{p_2p_5, p_4p_5\}$ and
$N_G(d_2)=\{p_1p_5, p_3p_5\}$, and so $S_1:=p_2p_5,d_1, p_4p_5$ and $S_2:=p_1p_5, d_2, p_3p_5$ must be subpaths of $C$. Since $d_1\cap d_2 \neq \emptyset$, then $S_1$ and $S_2$ are disjoint.
Let $\pp':=V(C)\setminus V(S_1\cup S_2)$. Then, $C\setminus \{d_1,d_2\}$ consists of two disjoint subpaths, say $S_3$ and $S_4$. Then $C=S_1\cup S_3 \cup S_2 \cup S_4$ and
moreover, the union of the inner vertices of $S_3$ and $S_4$ is precisely $\{p_1p_2, p_2p_3, p_3p_4, p_4p_1\} =\pp'$.

From the definitions of $S_1$ and $S_2$ it is easy to see that no endvertex of $S_1$ is adjacent in $G$ to any endvertex of $S_2$. Then, for $j\in \{3,4\}$,  $S_j$ must have at least one inner vertex. Let $S^{\circ}_j$ be the subpath of $S_j$ induced by its inner vertices. Since the subgraph of $G$ induced by any subset of $\pp'$ with at least 3 elements is disconnected, then $S^{\circ}_j$ must have exactly 2 vertices.  So, $S^{\circ}_3$ and $S^{\circ}_4$ are edges of $C$, and consequently, each of $S_3$ and $S_4$ is a path with $4$ vertices. On the other hand, since the only neighbour of
$p_1p_2$ (respectively, $p_2p_3$) in $\pp'$ is $p_3p_4$ (respectively, $p_1p_4$), then we can conclude that $\{S^{\circ}_3;S^{\circ}_4\}=\{p_1p_2,p_3p_4; p_2p_3,p_1p_4\}$.

Without loss of generality suppose $S^{\circ}_3=p_1p_2,p_3p_4$ and $S^{\circ}_4=p_2p_3,p_1p_4$. Then, $S^{\circ}_3$ can be extended to $S_3$ only in two ways:
 $S_3=p_4p_5,p_1p_2,p_3p_4,p_1p_5$ or $S_3=p_3p_5,p_1p_2,p_3p_4,p_2p_5$. The first (respectively, last) case implies that
 $p_2p_5$ and $p_3p_5$ (respectively, $p_1p_5$ and $p_4p_5$) are vertices of $S_4$. Since it is impossible to extend $S^{\circ}_4$ to a path (namely, $S_4$) with any of these pairs, we conclude that such a $C$ does not exist.

Suppose now that ($C2$) holds. Let $p_1, p_2,\ldots ,p_6$ be the points of $P$ and assume
w.l.o.g. that they are located as in Figure~\ref{fig:Ejem}~(a)-(b). Again, we derive a contradiction from the assumption that $G$ contains a spanning cycle $C$.
Let $\{U, M, L\}$ be the vertex partition of $V(G)$ defined by $U:=\{p_1p_4, p_2p_5, p_3p_6\}$, $M:=\{p_1p_2, p_1p_6, p_2p_3,\-p_3p_4, p_4p_5, p_5p_6\}$, and
$L:=\{p_1p_3, p_1p_5, p_2p_4, p_2p_6, p_3p_5,p_4p_6\}$. We note that the degree of a vertex in $U, M,$ and $L$ is $2,6,$ and $3$, respectively. See Figure~\ref{fig:Ejem}~(c).
The vertices of degree 2 imply that $S_1:=p_5p_6,p_1p_4,p_2p_3$,  $S_2:=p_1p_6,p_2p_5,p_3p_4$, and $S_3:=p_1p_2,p_3p_6,p_4p_5$ must be subpaths of $C$.

Let $e_{1,2}, e_{1,3}$, and $e_{2,3}$ be the edges of $G$ whose endvertices are indicated by Figure~\ref{fig:Ejem}~(c). For $i,j\in \{1,2,3\}$ and $i<j$, we let
$E_{i,j}$ be the set formed by the 4 edges of $G$ that go from an endvertex of $e_{i,j}$ to $M$. Note that if $e_{i,j}\notin C$ for some of these edges, then the existence of
$C$ and the fact each endvertex of $e_{i,j}$ has degree 3 in $G$ imply that each edge of $E_{i,j}$ belongs to $C$. On the other hand, a simple inspection of
Figure~\ref{fig:Ejem}~(c) reveals that $C_{i,j}:=E_{i,j}\cup S_i\cup S_j$ is a cycle of order $8$. Since $C$ cannot have proper cycles, we conclude
that each $e_{i,j}$ must be in $C$. Then $(i)$ {\em each vertex in $L$ is incident with exactly one edge of $C$ that goes to $M$}. Since $S_1, S_2,$ and
$S_3$ are subpaths of $C$, $(ii)$ {\em the edges of $C$ that go from $L$ to $M$ define a matching of size 6}.

Let $C':=S_1\cup S_2 \cup S_3 \cup \{e_{1,2}, e_{1,3}, e_{2,3}\}$. We know that $C'\subsetneq C$. Let $E(L,M)$ be the set of edges of $G$ that go from $L$ to $M$, and let us label these by $e_1, e_2,\ldots ,e_{12}$, as indicated by Figure~\ref{fig:Ejem}~(c). A pair $\{e_i,e_j\}\subset E(L,M)$ will be called {\em mutually exclusive}
(or {\em m.e.} for short), if the subgraph of $G$ induced by $C'\cup \{e_i,e_j\}$ has a cycle. Because $C$ cannot have proper cycles, $C\cap E(L,M)$ cannot contain an m.e. pair.

From $(i)$ we know that exactly one of $e_1\in C$ or $e_3\in C$ holds. Suppose first that $e_1\in C$ holds. Since $\{e_1,e_4\}$ is m.e., then $e_6\in C$ by $(i)$. This last and
$(ii)$ imply that $e_{10}\notin C$, and so $e_{12}\in C$ by $(i)$. From $e_{12}\in C$ and $(ii)$ it follows that $e_{11}\notin C$, and so $e_2\in C$ by $(i)$. As $e_1, e_2\in C$ contradicts $(ii)$, we conclude that $e_1\in C$ is impossible.  Suppose now that $e_3\in C$. Then $e_1\notin C$ by $(i)$, and so $e_2\in C$ by $(ii)$. This last and
$(ii)$ imply that $e_{11}\notin C$, and so $e_{12}\in C$ by $(i)$. From $e_{12}\in C$ and $(ii)$ it follows that $e_{10}\notin C$, and so $e_6\in C$ by $(i)$. Then,
$\{e_3, e_6\}$ is an m.e. pair of $C$, which is impossible. This last contradiction proves if ($C2$) holds, then $G$ is not hamiltonian.
\end{proof}



\subsection{$D(P)$ is hamiltonian for the rest of the instances of $P$}\label{s:computer} Let us denote by $C_6$ the order type corresponding to the set of
$6$ points in convex position, and let
 ${\mathbb S}$ be the collection of all order types such that $P\in {\mathbb S}$ if and only if $|P|\in \{6,7,8\}$ and $P\not\sim C_6$. From Table~\ref{table1} we know that
${\mathbb S}$ has exactly 3465 sets of points.  Then, it remains to check that $D(P)$ is hamiltonian for each $P\in {\mathbb S}$.

Motivated by the short proof of Theorem~\ref{thm:main} for $n\geq 9$, we were first looking for known results in the Hamilton cycles literature that help to settle the current case in a simple form.
After spent some time in an unsuccessful search, we found a very promising hamiltonicity criterion in~\cite{plotnikov}. Unfortunately, as we shall see in more detail in Section~\ref{s:plotnikov}, such a criterion fails for certain minor of $D(C_6)$. In a second effort, we worked in a constructive approach, and  we soon discovered an easy way to
construct explicit hamiltonian cycles of $D(P)$  for the case in which $P$ is in convex position and $n\geq 7$, however, we were unable to extend that technique to the general case.

Finally, we find the complete and reliable database for all possible order types of size at most 10 provided by
O. Aichholzer et al. in ~\cite{oswin1}. Each of these order types is stored and available in~\cite{oswin2} as a set of points with explicit positive integers coordinates. We first implement
an algorithm in Mathematica (Wolfram Language) whose input is the set of points defining $P\in {\mathbb S}$
and produces the corresponding graph $D(P)$, as output. Then, for each $P\in {\mathbb S}$ we verified the hamiltonicity of $D(P)$ with the help of the function that Mathematica has available for checking the hamiltonicity of small graphs \cite{wolfram}. According with the results obtained from that exhaustive verification we have the following lemma.

\begin{lem}\label{l:678}
If ${\mathbb S}$ is as above, then $D(P)$ is hamiltonian for $P\in {\mathbb S}$.
\end{lem}

\noindent{\em Proof of Theorem~\ref{thm:main}}. It follows immediately from Lemmas~\ref{l:nine},~\ref{l:negative}, and~\ref{l:678}. \hfill $\square$


\section{A counterexample to the hamiltonicity criterion claimed in~\cite{plotnikov}}\label{s:plotnikov}
Let us recall the concepts used in~\cite{plotnikov}. Let $H=(V(H), E(H))$ be a connected graph. If $X\neq \emptyset$ is an independent set of $H$,
then $\mathbf{P}(X)$ will be the set of all paths with both endvertices in $X$. A vertex subset $U$ of $\mathbf{P}(X)\setminus X$ is a {\em separator}
of $X$, if there are no two vertices of $X$ in a component of $H\setminus U$. Then any separator of  $X$ is a vertex cut of $H$.

\begin{stat}[Theorem 3.1~\cite{plotnikov}]\label{th:wrong}
The connected graph $H=(V(H), E(H))$ is Hamiltonian iff for any independent set $X\subset V(H)$, the following holds:
$|X|\leq \min\{|U|:~U \mbox{ is a separator of } X\}.$
\end{stat}

\begin{thm}\label{thm2}
The graph $H$ shown in Figure \ref{fig:grapg_G}  is a counterexample to Statement~\ref{th:wrong}.
\end{thm}
\begin{proof}
Let $H$ be the graph in Figure \ref{fig:grapg_G}. We note that $H$ is a $2$--connected non-hamiltonian graph, and  that $H$ has no independent sets of order 4.
Clearly, it is enough to show that $|X|\leq \min\{|U|:~U \mbox{ is a separator of } X\}$ for any independent set $X\subset V(H)$.

Let  $X$ be an independent set of $H$. Since $H$ is a $2$--connected and $\alpha(H)<4$, we can assume that $|X|=3$. Seeking a contradiction, suppose that there is
a separator $U$ of $X$ such that $|U|<|X|$. The $2$-connectivity of $H$ implies that $|U|=2$, and the definition of $U$ implies that $H\setminus U$ must have at least three components.
By simple inspection of $H$ we can conclude that there is no such $U$.
\end{proof}

As we mentioned in Section~\ref{s:computer}, in our search for some known result that allows us to simplify the proof of Theorem~\ref{thm:main} for $n\leq 8$ we found~\cite{plotnikov}, and when applying the given criteria to our small graphs we detected the inconsistency. After reading the proof of Statement~\ref{th:wrong}, we have seen that the origin of the error is a wrong interpretation of some results given in~\cite{hoede} (which involve contraction of edges).

The graph $H$ in Figure~\ref{fig:grapg_G} is a minor of the graph $D(P)$ shown in Figure \ref{fig:Ejem}~(c). In other words, $D(P)$
 contains a subset $E$ of edges such that  $H$ can be obtained from $D(P)\setminus E$ by a succession of edge contractions.

\begin{figure}[H]
	\centering
	\includegraphics[width=0.3\textwidth]{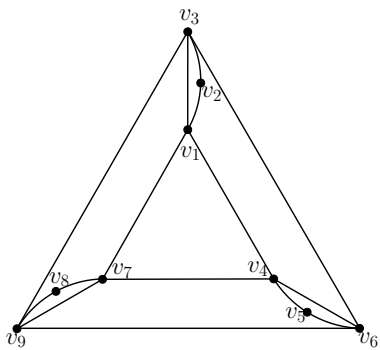}
	\caption{\small  A counterexample to Statement~\ref{th:wrong}.}
		\label{fig:grapg_G}
\end{figure}



\end{document}